\documentclass[11pt]{amsart}
\usepackage{amssymb,amsmath,amsthm}
\usepackage{enumitem}
\usepackage{tikz-cd}
 \usepackage{mathtools}
\usetikzlibrary{matrix,arrows,decorations.pathmorphing,shapes.geometric, calc}
\usepackage{verbatim}
\usepackage{psfrag}
\usepackage{bm}

\newtheorem{conjecture}[equation]{Conjecture}
\newtheorem{theorem}[equation]{Theorem}
\newtheorem{lemma}[equation]{Lemma}
\newtheorem{prop}[equation]{Proposition}
\newtheorem{cor}[equation]{Corollary}
\newtheorem{corollary}[equation]{Corollary}

\theoremstyle{remark}
\newtheorem{remark}[equation]{Remark}

\numberwithin{equation}{section}

\newcommand{\abs}[1]{\lvert#1\rvert}
\newcommand{\Sph}{\mathbb{S}}

\newcommand{\id}{\mathrm{Id}}

\newcommand{\N}{\mathbb{N}}

\newcommand{\R}{\mathbb{R}}
\newcommand{\Z}{\mathbb{Z}}
\newcommand{\B}{\mathbb{B}}
\newcommand{\D}{\mathbb{D}}

\newcommand{\ddiv}{\mathrm{div}}

\newcommand{\Ccal}{\mathcal{C}}

\newcommand{\Hscr}{\mathcal{H}}
\newcommand{\Tan}{\mathrm{Tan}}

\newcommand{\spt}{\mathrm{spt}\,}
\newcommand{\Mbold}{\bm{\mathrm{M}}}
\newcommand{\Dscr}{\mathcal{D}}
\begin{document}

\title[Area for genus zero FBMS]{On the areas of genus zero free boundary minimal surfaces embedded in the unit $3$-ball}

\author[P.~McGrath]{Peter~McGrath}
\author[J.~Zou]{Jiahua~Zou}
\date{}
\address[Peter McGrath]{Department of Mathematics, North Carolina State University, Raleigh NC 27695} 
\email{pjmcgrat@ncsu.edu}
\address[Jiahua Zou]{Department of Mathematics, Brown University, Providence RI 02912}
\email{jiahua\_zou@brown.edu}

\begin{abstract}
We prove that the area of each nonflat genus zero free boundary minimal surface embedded in the unit $3$-ball is less than the area of its radial projection to $\Sph^2$.  The inequality is asymptotically sharp, and we prove any sequence of surfaces saturating it 
converges weakly to $\Sph^2$, as currents and as varifolds.
\end{abstract}
\maketitle

\section{Introduction}

\subsection*{Background and main results}
\phantom{ab}

Let $M$ be a compact Riemannian manifold with nonempty boundary.  A minimal immersion $\varphi : \Sigma \rightarrow M$ is called a \emph{free boundary minimal surface} if $\Sigma$ meets $\partial M$ orthogonally along $\partial \Sigma$.  Such surfaces are area-critical for deformations of $M$ which preserve $\partial M$ as a set.

When $M$ is the euclidean unit ball $\B^n$, the class of free boundary minimal surfaces has particularly rich geometric properties.  Fraser-Schoen \cite{FS1} proved the sharp lower area bound $|\Sigma | \geq \pi$ for free boundary minimal $\Sigma \subset \B^n$---equality holds only if $\Sigma$ is an equatorial disk (see also \cite{Brendle:area})---while Fraser-Li proved the following upper area bounds: 

\begin{prop}[Fraser-Li {\cite[Prop. 3.4]{FraserLi}}]
\label{Pupper}
If $\Sigma \subset \B^3$ is an embedded free boundary minimal surface with genus $\gamma$ and $k$ boundary components, then
\begin{align}
\label{FLupper}
|\Sigma| \leq \mathrm{min} \left \{ 2(\gamma +k) \pi, 8 \pi \left \lfloor \frac{\gamma+3}{2}\right\rfloor\right\}.
\end{align}
\end{prop}

The bounds in  \eqref{FLupper} are not optimal, and one purpose of this paper is to prove a sharp upper bound when the genus $\gamma$ is zero.  

In surprising recent developments, several authors \cite{GL, KapZou, KarpukhinStern} have constructed sequences of genus zero free boundary minimal surfaces embedded in $\B^3$ which converge to the positively curved boundary $\partial \B^3 = \Sph^2$ in the sense of varifolds.  In particular, the areas of these surfaces limits to $4\pi$.  

As a corollary of the following stronger result, it follows that $4\pi$ is in fact the optimal upper bound: 

\begin{theorem}
\label{Tmain}
Let $\Sigma \subset \B^3$ be an embedded free boundary minimal surface with genus zero and at least two boundary components.  Then  
\begin{align*}
 | \Sigma | < |\Omega|,
\end{align*}
where $\Omega \subset \Sph^2$ is the radial projection of $\Sigma$ to $\Sph^2$. 
\end{theorem}

Because the only free boundary minimal disks are the totally geodesic equators \cite{Nitsche}, an immediate consequence is the following optimal area bound:

\begin{cor}
\label{Cmain}
Each genus zero free boundary minimal surface embedded in the unit $3$-ball has area strictly less than $4\pi$. 
\end{cor}

The sequences of surfaces converging weakly to the sphere from \cite{GL, KapZou, KarpukhinStern}  show the inequality in Theorem \ref{Tmain} is asymptotically sharp.    Our second main result confirms that this kind of weak convergence is the \emph{only} way the inequality can be saturated:

\begin{theorem}
\label{Tconv}
For each $n \in \N$, let $\Sigma_n \subset \B^3$ be an embedded free boundary minimal surface with genus zero and at least two boundary components, and let $\Omega_n \subset \Sph^2$ be its radial projection.
If  $| \Omega_n| - |\Sigma_n|\rightarrow 0$ as $n\rightarrow \infty$, then $\{\Sigma_n\}_{n \in \N}$ converges to $\Sph^2$, as varifolds and as currents. 
\end{theorem}

\begin{remark}
We note the following regarding Theorems \ref{Tmain} and \ref{Tconv}: 
\begin{itemize}
\item The embeddedness is essential: there are free boundary minimal annuli immersed in $\B^3$ with arbitrarily large area \cite{Fernandez, KapMcG}.  

\item By \cite[Theorem 1.2]{FraserLi}, no sequence of surfaces saturating the inequality in \ref{Tmain} has a bounded number of boundary components.

\item A result analogous to \ref{Tmain} where $\B^3$ is replaced by a hemisphere in the round $3$-sphere $\Sph^3$ is false: this can be seen by considering the area expansion \cite[Theorem A(v)]{KapMcGLDG} applied to doublings of the equatorial $2$-sphere \cite{kapsph, kapmcgsph} in $\Sph^3$.
\end{itemize}
\end{remark}
\subsection*{The general framework}
\phantom{ab}

The theory of free boundary minimal surfaces in $\B^n$ is intertwined with the spectral geometry of the Steklov eigenvalue problem.  A metric on a compact surface $\Sigma$ which maximizes $\sigma_1(\Sigma) |\partial \Sigma|$, where $\sigma_1(\Sigma)$ is the first Steklov eigenvalue, arises \cite{FS2} as the induced metric on a free boundary minimal surface in $\B^n$.  The free boundary minimal surfaces in \cite{GL, KarpukhinStern} mentioned above arise this way, and these surfaces saturate Kokarev's upper bound \cite{Kokarev}
\begin{align}
\label{Ekokarev}
\sigma_1(\Sigma) |\partial \Sigma| < 8\pi,
\end{align}
valid for any compact genus zero surface with boundary.  

On the other hand, if $\Sigma$ is also a free boundary minimal surface, applying the divergence theorem to the position vector field $V(x) = x$ establishes 
\begin{align}
\label{Epos}
2 | \Sigma| = \int_{\Sigma} \ddiv_\Sigma V = \int_{\partial \Sigma} \langle V, x\rangle = | \partial \Sigma|. 
\end{align}

Combining \eqref{Ekokarev} and \eqref{Epos} shows that the following well-known conjecture posed by Fraser-Li \cite{FraserLi} implies Corollary \ref{Cmain}:

\begin{conjecture}
 \label{CFL}
 Let $\Sigma$ be a properly embedded free boundary minimal hypersurface in the Euclidean unit ball $\B^n$.  Then $\sigma_1(\Sigma) = 1$.
\end{conjecture}

Conjecture \ref{CFL} is a variant of Yau's eigenvalue conjecture \cite{Yau:Problems}, asserting that the first Laplace eigenvalue on any embedded minimal hypersurface in the sphere $\Sph^{n+1}$ is $n$.  Conjecture \ref{CFL} has profound implications other than Corollary \ref{Cmain}, including an improvement in the bounds \eqref{FLupper} by a factor of $2$, and a characterization \cite{FraserSchoen} of the critical catenoid as the unique embedded free boundary minimal annulus in $\B^3$.  Another purpose of this paper is, by proving Corollary \ref{Cmain}, to give further evidence for the validity of Conjecture \ref{CFL}. 

\subsection*{Outline of the strategy and main ideas}
\phantom{ab}

The proof of Theorem \ref{Tmain} is inspired by seminal work of Harvey-Lawson on calibrated geometry, who showed \cite[Theorem 4.9]{HarveyLawson} that if an oriented Riemannian manifold is foliated by minimal hypersurfaces, then each leaf is area-minimizing in its homology class. 

To prove Theorem \ref{Tmain}, let $\Sigma \subset \B^3$ be a free boundary minimal surface, and consider the cone $\Ccal_\Sigma = \{t \Sigma : t \in (0, \infty), p \in \Sigma\} \subset \R^3$.  While it is not true in general, if $\Sigma$ is also embedded and of genus zero, a delicate nodal domain argument using the two-piece property \cite{LimaMenezes} shows that the dilates $t \Sigma$ of $\Sigma$ actually foliate the cone $\Ccal_\Sigma$.  Because the leaves are minimal, the calibration result above shows $\Sigma$ is area-minimizing in $\Ccal_\Sigma$.  In particular, $|\Sigma| < |\Omega|$ since the radial projection $\Omega$ is in $\Ccal_\Sigma$ and $\partial \Omega = \partial \Sigma$. 

The nodal domain argument above adapts an argument \cite[Prop. 8.1]{FraserSchoen} of Fraser-Schoen, and is also similar to arguments in Brendle's characterization \cite{Brendle:shrinker} of the round sphere as the unique compact embedded genus zero self-shrinker for the mean curvature flow.

The starting point for the proof of Theorem \ref{Tconv} is the ``tilt-excess" identity 
\begin{align}
\label{Ete}
|\Omega| - | \Sigma| = \frac{1}{2} \int_{\Omega} | \nu \circ \pi^{-1} - \nu_{\Sph^2}|^2, 
\end{align}
valid for $\Sigma$ and $\Omega$ as in Theorem \ref{Tmain}, where $\pi : \Sigma \rightarrow \Omega \subset \Sph^2$ is the radial projection, and $\nu$ and $\nu_{\Sph^2}$ are the outward pointing normal vectors to $\Sigma$ and $\Sph^2$ respectively.  The right hand side of \eqref{Ete} is strikingly similar to a tilt-excess quantity \cite[Lemma 8.13]{Allard} which plays an important role in the proof of the Allard's celebrated regularity theorem \cite{Allard}.

If $\Sigma_n$ is a sequence of surfaces as in Theorem \ref{Tconv}, we first use the tilt-excess identity \eqref{Ete} in conjunction with the radial graph property and the fact that $\Sigma_n$ minimizes area in its cone $\Ccal_{\Sigma_N}$ to conclude that the Hausdorff distance between $\Sigma_n$ and its radial projection $\Omega_n \subset \Sph^2$ goes to zero as $n \rightarrow \infty$. 

We next consider the rectifiable currents $T_n$ associated to the surfaces $\Sigma_n$.  By a standard compactness result \cite[4.2.17]{Federer} of Federer, $T_n$ converges subsequentially to a rectifiable current $T$.
Because the Hausdorff distance between $\Sigma_n$ and $\Omega_n$ goes to zero, it follows that $T$ is supported in $\Sph^2$. 
Using the convergence properties already established in conjunction with both the minimality and the free boundary condition, we are able to show that $\partial T = 0$.  It then follows from the constancy theorem \cite[4.1.31]{Federer} that the current $T$ is represented by integration over $\Sph^2$, taken with some integer multiplicity.  Using the radial graph property and Theorem \ref{Tmain}, we argue that this multiplicity must be one, and it follows that $\Sigma_n \rightarrow \Sph^2$ in the sense of currents. 

Finally, we use \eqref{Ete} and the convergence in the sense of currents to complete the argument that $\Sigma_n \rightarrow \Sph^2$ as varifolds.

\subsection*{Acknowledgments}
\phantom{ab}

We thank Daniel Stern for a helpful conversation, Nikolaos Kapouleas, for years of support and guidance, and Joseph Hoisington, for pointing out an error in an earlier version of this paper.  We are also indebted to Misha Karpukhin, for encouraging us to prove Theorem \ref{Tconv}, and for sharing with us a partial outline for its proof. 

\section{Proof of Theorem \ref{Tmain}}
We will need the the following result from \cite{KusnerMcGrath}.  We sketch a proof to keep the article self-contained; for more details we refer the interested reader to \cite[Prop. 8.1]{FraserSchoen} and \cite[Lemma 4.2]{Seo}.

\begin{lemma}[Radial graphs]
\label{Cfs}
Let $\Sigma \subset \B^3$ be an embedded free boundary minimal surface of genus zero and at least two boundary components.  Then: 
\begin{enumerate}[label=\emph{(\roman*)}]
\item $\Sigma$ is a radial graph, in the sense that $0\notin \Sigma$ and each ray from $0$ transversally intersects $\Sigma$ at most once; 
\item each component of $\partial \Sigma$ is a convex curve in $\Sph^2$.
\end{enumerate}
\end{lemma}
\begin{proof}
By the two-piece property \cite{LimaMenezes}, each  $2$-dimensional linear subspace separates $\Sigma$ into exactly two pieces.  It follows that no such subspace meets the interior of $\Sigma$ tangentially, for otherwise the intersection would contain at least two arcs meeting at a point of tangential contact, and would---since $\Sigma$ has genus zero---separate $\Sigma$ into more than two pieces.
Consequently, $0 \notin \Sigma$ and the differential $d\pi_p$ of the radial projection $\pi : \Sigma \rightarrow \Omega$ is injective for each $p \in \Sigma \setminus \partial \Sigma$.  Thus $\pi$ is a local diffeomorphism by the inverse function theorem.

Again using the two-piece property and the fact that $\Sigma$ has genus zero, a nodal domain argument similar to the one above shows that a $2$-dimensional subspace which meets a component of $\partial \Sigma$ transversely meets it at exactly two points, and item (ii) follows. 

To prove that $\pi$ is a global homeomorphism, let $\varphi : M \rightarrow \Sigma$ be a parametrizing immersion, where the surface $M$ is a subset of $\Sph^2$.  Each component $S \subset \partial M$ bounds a disk $D_S \subset \Sph^2 \setminus M$, and the preceding shows that $\varphi(S)$ bounds a convex disk $D_{\varphi(S)} \subset \Sph^2$.  Therefore $\varphi$ can be extended to a local homeomorphism $\Phi : \Sph^2 \rightarrow \Sph^2$ with the property that $\Phi(D_S) = D_{\varphi(S)}$ for each component $S$ of $\partial M$.  Since $\Sph^2$ is simply connected, the local homeomorphism $\Phi$ is actually a global homeomorphism, completing the proof.
\end{proof}

\begin{theorem}
\label{Tcal}
Let $\Sigma \subset \B^3$ be an embedded free boundary minimal surface with genus zero.  Then $\Sigma$ is area-minimizing in its cone
\[
\Ccal_\Sigma = \{ t p : t \in (0, \infty), p \in \Sigma\} \subset \R^3.
\]
\end{theorem}
\begin{proof} 
If $\Sigma$ is a disk, it is totally geodesic by \cite{Nitsche} and the conclusion to the Theorem holds trivially.  Therefore, we may assume $\Sigma$ has at least two boundary components. 

By the radial graph property, $\Ccal_\Sigma$ is foliated by the dilates $t \Sigma$ of $\Sigma$.  Let $\nu$ be the unit outward pointing normal vector field on $\Sigma$, and extend $\nu$ to a unit smooth vector field (with the same name) on $\Ccal_\Sigma$ by requesting $\nu$ is normal to each leaf $t \Sigma$. 

For any $p \in \Ccal_\Sigma$, let $\{e_1, e_2, e_3\}$ be an orthonormal frame for $T_p\R^3$ with the property that $e_3 = \nu$.  Then 
\begin{align*}
\ddiv\,  \nu = \sum_{i=1}^2 \langle D_{e_i} \nu, e_i \rangle + \langle D_\nu \nu, \nu \rangle
= H_{t \Sigma} + 0 = 0, 
\end{align*}
where we have used the minimality of $t \Sigma$ and that $ \langle D_\nu \nu, \nu \rangle = \frac{1}{2} \nu \langle \nu , \nu \rangle = 0$, since $\nu$ is a unit vector field. 

Let $\omega$ be the volume form on $\Ccal_\Sigma$.  Then the $1$-form $\varphi := i_\nu \omega$ on $\Ccal_\Sigma$ is closed because $d \varphi = \ddiv\,  \nu \, \omega = 0$, so Stokes' theorem implies that for any oriented surface $\Sigma' \subset \Ccal_\Sigma$  homologous to $\Sigma$, 
\begin{align}
\label{Ecalb}
|\Sigma| = \int_{\Sigma} \varphi = \int_{\Sigma'} \varphi = \int_{\Sigma'} \langle \nu, \nu_{\Sigma'}\rangle \leq |\Sigma'|.
\end{align}

\end{proof}

As a corollary, we have the following result, which immediately implies Theorem \ref{Tmain}:
\begin{corollary}[Tilt-excess identity]
\label{Ctmain}
For $\Sigma$ and $\Omega$ as in Theorem \ref{Tmain}, 
\begin{align}
\label{Etiltexcess}
|\Omega| - | \Sigma| = \frac{1}{2} \int_{\Omega} | \nu \circ \pi^{-1} - \nu_{\Sph^2}|^2, 
\end{align}
where $\pi : \Sigma \rightarrow \Omega$ is the radial projection, and $\nu$ and $\nu_{\Sph^2}$ are the normals to $\Sigma$ and $\Omega$, respectively. 
\end{corollary}
\begin{proof}
Since $\Sigma$ is a free boundary minimal surface, $\partial \Sigma = \partial \Omega$.  Furthermore, the radial graph property implies $\Omega \subset \Ccal_\Sigma$.  The conclusion now follows by applying \eqref{Ecalb} with $\Omega$ taking the role of $\Sigma'$,  adding and subtracting $|\Omega| = \int_{\Omega }1$, and observing that $1 - \langle \nu , \nu_{\Sph^2} \rangle = \frac{1}{2} | \nu - \nu_{\Sph^2}|^2$. 
\end{proof}

\begin{remark}
Theorem \ref{Tmain} can also be proved directly from the radial graph property, by
parametrizing $\Sigma$ as a graph over its projection $\Omega \subset \Sph^2$, and we sketch the details here.  Let $u \in C^2(\Omega)$ be the function such that  $X: \Omega \rightarrow \R^3$ given by $X(p) = p - u(p) p$ parametrizes $\Sigma$.

By Lemma \ref{LH}, the mean curvature of $X(\Omega) = \Sigma$ satisfies
\begin{align}
\label{EHmin}
0=(1-u)H = \ddiv \frac{\nabla u}{\sqrt{(1-u)^2 + |\nabla u|^2}} + \frac{2(1-u)}{\sqrt{(1-u)^2 + |\nabla u|^2}}. 
\end{align}
Integrating \eqref{EHmin} over $\Omega_{\varepsilon} : = \{ p \in \Omega : dist(p, \partial \Omega) > \varepsilon\}$, using the divergence theorem, and bounding the last term in \eqref{EHmin} by $2$, we conclude 

\begin{align*}
- \int_{\partial \Omega_{\varepsilon}} \frac{\langle \nabla u, \nu_{\varepsilon} \rangle}{\sqrt{(1-u)^2 + |\nabla u|^2}}
<  2|\Omega_\varepsilon|,
\end{align*}
where $\nu_{\varepsilon}$ is the outward pointing unit normal to $\Omega_\varepsilon$ along $\partial \Omega_\varepsilon$. But $u \searrow 0$ and $\nu_{\varepsilon} = - \nabla u / |\nabla u| + o(1)$ as $\varepsilon \searrow 0$, so we get $|\partial \Sigma| < 2 | \Omega|$ in the limit.
The conclusion now follows from using that $|\partial \Sigma| = 2|\Sigma|$, from \eqref{Epos}. 
\end{remark}

\section{Proof of Theorem \ref{Tconv}}
\label{Sconv}
Throughout this section, let $\{\Sigma_n\}_{n \in \N}$ be a sequence of embedded free boundary minimal surfaces in $\B^3$, each with genus zero and at least two boundary components.  Associated to $\Sigma_n$, let $\pi_n : \Sigma_n \rightarrow \Sph^2$ be its radial projection map, $\iota_n : \Sigma_n \rightarrow \R^3$ its inclusion map, $\nu_n$ its outward pointing unit normal, $da_n$ its area form, and let $\Omega_n = \pi_n(\Sigma_n)$.  As in Theorem \ref{Tconv}, we assume that
\begin{align}
\label{Emconv}
|\Omega_n | - | \Sigma_n | \rightarrow 0
\quad
\text{as}
\quad
n \rightarrow \infty. 
\end{align}

\begin{lemma}
\label{Lproj}
Let $\Sigma \subset \R^3$ be an embedded surface not containing the origin.  For any $p \in \Sigma$, the differential of the radial projection $\pi: \Sigma \rightarrow \Sph^2$ satisfies
\begin{align}
\label{Edetdp}
d\pi_p = \frac{1}{|p|} \Pi_{T_{\pi(p)}\Sph^2}, 
\quad
\det( d\pi_p) = \frac{\langle \nu(p), \nu_{\Sph^2} \circ \pi(p)\rangle}{|p|^2},
\end{align}
where $\Pi_{T_{\pi(p)}\Sph^2}$ is the orthogonal projection to $T_{\pi(p)} \Sph^2$. 
\end{lemma}
\begin{proof}
Straightforward calculation.
\end{proof}

\begin{lemma}
\label{Lconvbd}
The Hausdorff distance $d_{\Hscr}(\Sigma_n, \Omega_n)$ tends to $0$ as $n\rightarrow \infty$.
\end{lemma}
\begin{proof}
Fix a small number $\delta>0$. By Theorem \ref{Tcal}, $\Sigma_n$ is area-minimizing in its cone $\Ccal_{\Sigma_n}$, hence stable; 
consequently, curvature estimates \cite{Schoen} imply that the second fundamental form of each $\Sigma_n$ is bounded on $B_{1-\delta}: = \{ |x| < 1-\delta\}$.  Since $|\Sigma_n | < 4\pi$, it follows from \cite[Proposition 7.14]{CM} that a subsequence of $\Sigma_n$ converges smoothly to a smooth minimal surface $\Sigma$, possibly with multiplicity. 

To complete the proof of the Lemma, it suffices to prove that $\Sigma$ must be empty.  For the sake of a contradiction, suppose that $\Sigma$ is nonempty. 

Pulling back the identity \eqref{Etiltexcess} to $\Sigma_n$ using \eqref{Edetdp} reveals that
\begin{align}
\label{Etilt2}
|\Omega_n| - |\Sigma_n| = \frac{1}{2} \int_{\Sigma_n} \frac{\langle \nu_n , \nu_{\Sph^2}\circ \pi_n\rangle}{|p|^2} |\nu_n - \nu_{\Sph^2} \circ \pi_n|^2.
\end{align}

By taking the limit as $n \rightarrow \infty$, it follows that either $\langle \nu_{\Sigma} , \nu_{\Sph^2}\circ \pi\rangle \equiv 1$ or $\langle \nu_{\Sigma} , \nu_{\Sph^2}\circ \pi\rangle \equiv 0$ on $\Sigma$.  The former alternative implies that $\Sigma$ is contained in a sphere, which would contradict its minimality.  On the other hand, the latter alternative implies that the minimal surface $\Sigma$ is a cone.  Its link is then minimal in $\Sph^2$ by \cite[Prop 6.1.1]{Simons}, hence is a great circle.  Therefore $\Sigma$ is an equatorial disk $\D \subset \B^3$. 

Let $k \in \N$ be the multiplicity of the convergence $\Sigma_n \rightarrow \Sigma$.  We first prove $k \in \{1, 2\}$ by arguing as in \cite[Theorem 8.2]{FraserSchoen}: if $k\geq 3$, then the $1$-dimensional subspace orthogonal to $\D$ would meet $\Sigma_n$ at least three points for all large enough $n$. At least one of the two corresponding rays would then meet $\Sigma_n$ at at least two points, contradicting the radial graph property.

To rule out the cases $k=1$ and $k=2$, we derive a contradiction with \eqref{Emconv} by estimating the corresponding areas inside and outside of the cone
\begin{align*}
\Ccal^{1-\delta}_{\Sigma_n}  = \{ tp : t \in (0, \infty), p \in \Sigma_n \cap B_{1-\delta}\}.
\end{align*}

First suppose $k=1$.  By the radial graph property, each ray in $\Ccal^{1-\delta}_{\Sigma_n}$ meets $\Sigma_n$ exactly once.  Then $\Sigma_n \cap \Ccal^{1-\delta}_{\Sigma_n}$ converges smoothly and uniformly to $\D \cap B_{1-\delta}$ with multiplicity one.  Hence for all large enough $n$,  
\begin{align}
\label{Esigman}
| \Sigma_n | = |\Sigma_n \cap \Ccal^{1-\delta}_{\Sigma_{n}}| + |\Sigma_n \setminus \Ccal^{1-\delta}_{\Sigma_n}| < \pi +  |\Sigma_n \setminus \Ccal^{1-\delta}_{\Sigma_n}| +o(1).
\end{align}

We now estimate $|\Omega_n|$. 
Since each ray in $\Ccal^{1-\delta}_{\Sigma_n}$ meets $\Sigma_n$ exactly once by the radial graph property, $|\Omega_n \cap \Ccal^{1-\delta}_{\Sigma_n}| > 2\pi -C \delta$ for some $C> 0$. 
To estimate $|\Omega \setminus \Ccal^{1-\delta}_{\Sigma_n}|$, note that $\Sigma_n\setminus \Ccal^{1-\delta}_{\Sigma_n}$, $\Omega_n \setminus \Ccal^{1-\delta}_{\Sigma_n}$, and $\partial \Ccal^{1-\delta}_{\Sigma_n}$ bound a domain in $\Ccal_{\Sigma_n} \setminus B_{1-\delta}$.  Since $\Sigma_n$ is area-minimizing in $\Ccal_{\Sigma_n}$ by Theorem \ref{Tcal}, 
\begin{align*}
|\Sigma_n \setminus \Ccal^{1-\delta}_{\Sigma_n} | < |\Omega_n \setminus \Ccal^{1-\delta}_{\Sigma_n}| + C\delta,
\end{align*}
where we have used that the portion of $\partial \Ccal^{1-\delta}_{\Sigma_n}$ outside $B_{1-\delta}$ has area less than $C \delta$. 

Combining the preceding, we find
\begin{align}
\label{Eomegan}
|\Omega_n| > 2\pi + |\Sigma \setminus \Ccal^{1-\delta}_{\Sigma_n}| - C\delta.
\end{align}
Taken together, \eqref{Esigman} and \eqref{Eomegan} contradict \eqref{Emconv}.

Now suppose $k=2$.  We derive a contradiction in a very similar way to the $k=1$ case. 
By taking $n$ large enough, $\Sigma_n \cap \Ccal^{1-\delta}_{\Sigma_n}$ is now a pair of graphs with arbitrarily small $C^1$ norms.  Arguing as before, we have
\begin{align*}
|\Sigma_n| &< 2\pi + |\Sigma_n \setminus \Ccal^{1-\delta}_{\Sigma_n}|+ o(1), \\
| \Omega_n| &> 4\pi + |\Sigma_n \setminus \Ccal^{1-\delta}_{\Sigma_n}|- C\delta,
\end{align*}
which together contradict the assumption that $|\Omega_n| - |\Sigma_n| \rightarrow 0$.
\end{proof}

Let $T_n$ denote the current associated to $\Sigma_n$.  By Theorem \ref{Tmain} and \eqref{Epos},
\begin{align*}
\Mbold (T_n) < 4\pi, 
\quad
\Mbold (\partial T_n)  = 2 \Mbold (T_n) < 8\pi,
\end{align*}
so by the compactness theorem for integral currents \cite[4.2.17]{Federer}, a subsequence of $T_n$, which we may take to be the original sequence, converges to an integral current $T$.  Let $T'_n$ denote the current associated to $\Omega_n$.

\begin{prop}[Properties of $T$]The following hold. 
\label{LT}
\begin{enumerate}[label=\emph{(\roman*)}]
\item $\spt T \subset \Sph^2$.
\item $\Tan^2( \spt T , p) = T_p \Sph^2$ for $\Hscr^2$-almost every $p \in \spt T$. 
\item $T_n - T'_n \rightarrow 0$ as $n \rightarrow \infty$. 
\item $\partial T= 0$.
\item $T = ( \Hscr^2 \llcorner \Sph^2) \wedge \zeta$, where $\zeta$ is the standard $2$-vector field orienting $\Sph^2$.  In other words, $T$ is represented by integration over $\Sph^2$.
\end{enumerate}
\end{prop}
\begin{proof}
Item (i) follows from Lemma \ref{Lconvbd}, and item (ii) follows from (i) in conjunction  with \cite[3.2.18, 3.2.19]{Federer}.

For (iii), fix a $2$-form $\varphi \in \Dscr^2(\R^3)$.  In a neighborhood of $\Sph^2$ in $\R^3$, it is easy to see that there is a smooth function $f$ and smooth $1$-form $\xi$ such that
\begin{align*}
\varphi = f \, \pi^* \omega + dr \wedge \xi,
\end{align*}  
where $\omega$ is the area form on $\Sph^2$ and $r$ is the distance from the origin. 

Note that $\iota_n^* \pi^* \omega = \pi^*_n \omega =  \det(d\pi_n) d a_n$, so 
\begin{align*}
T_n ( f \pi^* \omega)=
\int_{\Sigma_n} f \pi^* \omega = \int_{\Sigma_n} f \det(d\pi_n) da_n = \int_{\Omega_n}  f \circ \pi^{-1}_n \omega
\end{align*}
and hence
\begin{align*}
(T_n - T'_n) ( f \pi^*\omega) = \int_{\Omega_n} ( f\circ \pi^{-1}_n - f) \omega . 
\end{align*}
From this and Lemma \ref{Lconvbd}, it follows that $(T_n- T'_n) (f \pi^* \omega) \rightarrow 0$. 

Next, because $T_n \rightarrow T$, items (i) and (ii) imply that
\begin{align*}
\lim_{n\rightarrow \infty} T_n (dr \wedge \xi) = T(dr \wedge \xi) = T'_n(dr \wedge \xi)=0,
\end{align*}
where we have used that $dr|_{\Sph^2} = 0$.
By combining the preceding with Lemma \ref{Lconvbd}, it follows that $(T_n - T'_n )\varphi \rightarrow 0$, completing the proof of (iii).

For (iv), fix a $1$-form $\xi \in \Dscr^1(\R^3)$.   For any $\varphi \in \Dscr^1(\R^3)$ satisfying $\iota^* \varphi = 0$, where $\iota: \Sph^2 \rightarrow \R^3$ is the inclusion map, $0 = d (\iota^* \varphi) = \iota^* d \varphi$, so (i)-(ii) imply 
 \begin{align*}
 T\,  d \xi  = T \, d ( \xi + \varphi).
 \end{align*}

Consequently, we may assume that $\xi = \pi^* (i_X\omega)$ on a neighborhood of $\Sph^2$ in $\R^3$,   where $\pi : \R^3 \setminus \{0\} \rightarrow \Sph^3$ is the radial projection,
$X$ is a vector field on $\Sph^2$, and $\omega$ is the area form on $\Sph^2$.  We have
\begin{align*}
d \xi = d ( \pi^* i_X \omega ) = \pi^* ( d (i_X \omega )) = \pi^*( \ddiv_{\Sph^2}X \, \omega).
\end{align*}
It follows that
\begin{align*}
\iota_n^* d \xi = (\ddiv_{\Sph^2} X )\circ \pi_n \det (d\pi_n) da_n.
\end{align*}
By extending $X$ to $\R^3 \setminus \{0\}$ by $\pi^*$, so that $X = X \circ \pi$, we have finally that
\begin{align*}
\iota_n^* d \xi
=
(\ddiv_{\R^3} X) \det (d \pi_n) da_n.
\end{align*}

Integrating over $\Sigma_n$ and using that $\ddiv_{\R^3} X = \ddiv_{\Sigma_n} X+ \langle D_{\nu_n} X , \nu_n\rangle$, it follows that $T_n \, d \xi = \int_{\Sigma_n} \iota^*_nd \xi  = (I) + (II) + (III)$, where
\begin{equation*}
\begin{gathered}
(I) =  \int_{\Sigma_n} \ddiv_{\Sigma_n}  X (\det(d\pi_n)-1) da_n, 
\quad
(II) = \int_{\Sigma_n} \ddiv_{\Sigma_n} X da_n,\\
(III) =  \int_{\Sigma_n} \langle D_{\nu_n} X, \nu_n \rangle \det (d\pi_n) da_n.
\end{gathered}
\end{equation*}

From the divergence theorem, the minimality of $\Sigma_n$, the free boundary condition, and that $X|_{\Sph^2}$ is tangential to $\Sph^2$, we find 
\begin{align*}
(II) = \int_{\Sigma_n} \ddiv_{\Sigma_n} X da_n = \int_{\partial \Sigma_n} \langle X, x\rangle da_n = 0.
\end{align*}

The assumption that $|\Omega_n| - | \Sigma_n| \rightarrow 0$ together with \eqref{Etiltexcess} implies that 
\begin{align*}
| \nu_n \circ \pi^{-1}_n - \nu_{\Sph^2} | \chi_{\Omega_n} \rightarrow 0 
\quad
\text{in}
\quad
L^2(\Sph^2),
\end{align*}
hence converges to zero almost everywhere along a subsequence.  Combining this with Lemma \ref{Lproj} and Lemma \ref{Lconvbd} implies that $(1- \det(d \pi^{-1}_n) )\chi_{\Omega_n} \rightarrow 0$ almost everywhere along a subsequence.  Consequently, $|(I)| \rightarrow 0$ by changing variables to write the integral over $\Omega_n$ and using the dominated convergence theorem.  

Arguing along the same lines using that $X = X\circ \pi$ shows that $|(III)|\rightarrow 0$ as $n\rightarrow \infty$.  Thus $0 = \lim T_n\, d \xi_n = T  d \xi =  \partial T\,  \xi$, completing the proof of (iv). 

For (v), item (iv) and the constancy theorem \cite[4.1.31]{Federer} show that
\begin{align*}
 T = k( \Hscr^2 \llcorner \Sph^2) \wedge \zeta
\end{align*} 
for some $k \in \R$, where $\zeta$ is the standard unit $2$-vector field orienting $\Sph^2$; since $T$ is an integral current it follows moreover that $k \in \Z$. 

By combining lower semicontinuity of mass and Theorem \ref{Tmain}, 
\begin{align*}
\Mbold(T) \leq \liminf \Mbold(T_n) \leq 4\pi,
\end{align*}
so $|k| \leq 1$.  
On the other hand, item (iii) and the fact that $T_n \rightarrow T$ mean that $T'_n \rightarrow T$.  Let $\omega \in \Dscr^2(\R^3)$ be a $2$-form restricting to the area form on $\Sph^2$.  Since both $T$ and $T'_n$ are supported in $\Sph^2$, 
\begin{align}
\label{Earealim}
\lim_{n \rightarrow \infty} \Mbold(T'_n) = \lim_{n \rightarrow \infty} T'_n \omega = T\omega = \Mbold(T).
\end{align}
But $\Mbold(T'_n ) \geq  \Mbold(T_n) \geq \pi$ from Theorem \ref{Tmain} and \cite[Theorem 5.4]{FS1}, 
so $\Mbold(T) \geq \pi$. 
Therefore $k=1$ and the proof is complete. 
\end{proof}

\begin{proof}[Proof of Theorem \ref{Tconv}]
It was shown in Proposition \ref{LT}(v) that $\Sigma_n$ converges to $\Sph^2$ in the sense of currents, so it remains to prove the varifold convergence. 

We must show that for any $\Phi \in C(\overline{\B^3} \times \Sph^2)$, 
\begin{align}
\label{Evar}
\left| \int_{\Sigma_n} \Phi(x, \nu_n) - \int_{\Sph^2}\Phi(x, \nu_{\Sph^2})\right| \rightarrow 0
\quad 
\text{as} 
\quad n \rightarrow \infty.
\end{align}
Fix such a $\Phi$.  We have
\begin{align*}
\int_{\Sigma_n} \Phi(x, \nu_n) = \int_{\Omega_n} \det(d\pi^{-1}_n)\Phi( \pi^{-1}_n(x), \nu_n \circ \pi_n^{-1}) .
\end{align*}
Thus $\int_{\Sigma_n} \Phi(x, \nu_n) - \int_{\Sph^2}\Phi(x, \nu_{\Sph^2}) = (I) + (II) + (III)+(IV)$, where
\begin{align*}
(I) &= \int_{\Omega_n} \left(\det ( d\pi_n^{-1}) - 1\right)\Phi( \pi_n^{-1}(x), \nu_n \circ \pi_n^{-1}) , \\
(II) &= \int_{\Omega_n}\Phi( \pi_n^{-1}(x), \nu_n \circ \pi_n^{-1}) - \Phi( x, \nu_n \circ \pi_n^{-1}), \\
(III) &= \int_{\Omega_n}\Phi( x, \nu_n \circ \pi_n^{-1}) - \Phi(x,\nu_{\Sph^2}), \\
(IV) &= \int_{\Omega_n} \Phi(x, \nu_{\Sph^2}) - \int_{\Sph^2} \Phi(x, \nu_{\Sph^2}).
\end{align*}

The assumption that $|\Omega_n| - |\Sigma_n| \rightarrow 0$ together with \eqref{Etiltexcess} implies  that  
\begin{align}
\label{EL1}
| \nu_n \circ \pi^{-1}_n - \nu_{\Sph^2} | \chi_{\Omega_n} \rightarrow 0 
\quad
\text{in}
\quad
L^2(\Sph^2),
\end{align}
hence converges to zero almost everywhere along a subsequence.  Combining this with Lemma \ref{Edetdp} and Lemma \ref{Lconvbd} shows that $(\det(d\pi_n^{-1})-1) \chi_{\Omega_n} \rightarrow 0$ almost everywhere along a subsequence. Consequently, $|(I)| \rightarrow 0$ by the dominated convergence theorem.

The second and third terms are handled similarly: from Lemma \ref{Lconvbd} we have $\pi^{-1}_n - \id_{\Omega_n} \rightarrow 0$ uniformly in $n$, so $|(II)| \rightarrow 0$.  Next, \eqref{EL1} implies that $\nu_n \circ \pi_n^{-1} \chi_{\Omega_n} \rightarrow \nu_{\Sph^2}\chi_{\Omega_n}$ almost everywhere along a subsequence, so $|(III)| \rightarrow 0$ by the dominated convergence theorem.

Finally, we estimate
\begin{align*}
|(IV)| \leq \| \Phi \|_{L^\infty} | \Sph^2 \setminus \Omega_n|.
\end{align*}
It follows from Proposition \ref{LT}(v) and \eqref{Earealim} that $|\Omega_n| \rightarrow |\Sph^2|$, so $|(IV)| \rightarrow 0$ and the proof is complete. 
\end{proof}

\appendix
\section{Mean curvature for radial graphs}
\begin{lemma}
\label{LH}
If $\Omega \subset \Sph^2$ is a domain, $u \in C^2(\Omega)$, and the map $X : \Omega \rightarrow \R^3$ defined by $X(p) = p - u(p)p$
is an immersion, its mean curvature satisfies
\begin{align*}
(1-u)H = \ddiv \frac{\nabla u}{\sqrt{(1-u)^2 + |\nabla u|^2}} + \frac{2(1-u)}{\sqrt{(1-u)^2 + |\nabla u|^2}}. 
\end{align*}
\end{lemma}
\begin{proof}
Fix $p \in \Omega$, choose a normal coordinate system for $\Omega$ centered at $p$, and use subscripts ``1, 2" to denote partial derivatives with respect to these coordinates.  In particular, $\{p, p_1, p_2\}$ is an orthonormal frame for $T_p\R^3$.   

We have $X_i = (1-u)p_i - u_i p$ for $i=1, 2$, so the normal $\nu$ to $\Sigma$ is
\begin{align*}
\nu = \frac{(1-u)p +\nabla u}{\sqrt{(1-u)^2+\abs{\nabla u}^2}}.
\end{align*}
The coefficients of the first fundamental form for $\Sigma$ are
\begin{equation*}
	E=u_1^2+(1-u)^2, \quad F=u_1u_2, \quad G=u_2^2+(1-u)^2,
\end{equation*}
and the coefficeints of the second fundamental form for $\Sigma$ are
\begin{align*}
\begin{gathered}
	L=\frac{((1-u)+u_{11})(1-u)+2u_1^2}{\sqrt{(1-u)^2+\abs{\nabla u}^2}},
	\quad 
	 M=\frac{u_{12}(1-u)+2u_1u_2}{\sqrt{(1-u)^2+\abs{\nabla u}^2}}, \\
	 N=\frac{((1-u)+u_{22})(1-u)+2u_2^2}{\sqrt{(1-u)^2+\abs{\nabla u}^2}}.
\end{gathered}
\end{align*}
Therefore, 
\begin{multline*}
	(1-u)H=(1-u)\frac{EN-2MF+GL}{EG-F^2}\\
\begin{aligned}
	& =\frac{(u_1^2+(1-u)^2)u_{22}+(u_2^2+(1-u)^2)u_{11}-2u_1u_2u_{12}}{((1-u)^2+\abs{\nabla u}^2)^{3/2}}\\
	 &\phantom{=}+ \frac{3\abs{\nabla u}^2(1-u)+2(1-u)^3}{((1-u)^2+\abs{\nabla u}^2)^{3/2}}\\
	& = \frac{(\abs{\nabla u}^2+(1-u)^2)\Delta u-(u_{11}u_1^2+u_{22}u_2^2+2u_{12}u_1u_2)+\abs{\nabla u}^2(1-u)}{((1-u)^2+\abs{\nabla u}^2)^{3/2}}\\
	 &\phantom{=}+\frac{2(1-u)}{\sqrt{(1-u)^2+\abs{\nabla u}^2}}\\
	&= \mathrm{div}\frac{\nabla u}{\sqrt{(1-u)^2+\abs{\nabla u}^2}}+\frac{2(1-u)}{\sqrt{(1-u)^2+\abs{\nabla u}^2}}.
\end{aligned}
\end{multline*}
	\end{proof}

\end{document}